\documentclass[a4paper] {article} 
\usepackage{latexsym,amsmath,enumerate,amssymb,amsbsy,amsthm,enumerate,graphicx,xcolor,xparse,leftidx,indentfirst} 
\usepackage{float} 
\usepackage{titlesec}
\usepackage{pb-diagram}
\usepackage{verbatim}
\usepackage[stable]{footmisc}
\NewDocumentCommand{\ceil}{s O{} m}{%
  \IfBooleanTF{#1} 
    {\left\lceil#3\right\rceil} 
    {#2\lceil#3#2\rceil} 
}
 
\def\@biblabel#1{}

\renewcommand\footnotemark{}
\date{}
\newtheorem{thm}{Theorem}[section]
\newtheorem{cor}[thm]{Corollary}
\newtheorem{lem}[thm]{Lemma}
\newtheorem{prop}[thm]{Proposition}
\theoremstyle{definition}
\newtheorem{defn}[thm]{Definition}
\newtheorem{rem}[thm]{Remark}

\newcommand{\al}{\alpha}
\newcommand{\ze}{\zeta}

\newcommand{\D}{\Delta}
\newcommand{\la}{\lambda}
\numberwithin{equation}{section}
\numberwithin{table}{section}
\numberwithin{figure}{section}

\def\mathbi#1{{#1}}

\title{\textbf {\Large A Homology Theory of Graphs: First Homology Group of Hamiltonian Graphs}}
\author{\footnotesize PONGDATE MONTAGANTIRUD AND NATTHAWUT PHANACHET}

\def\mathbi#1{{#1}}

\AtEndDocument{\bigskip{\footnotesize%
  \textsc{Department of Mathematics and Computer Science, Chulalongkorn University, Bangkok, Thailand 10330} \par  
  P. Montagantirud, \textit{E-mail address}: \texttt{pongdate.m@chula.ac.th} \par
  \addvspace{\medskipamount}
  \textsc{Department of Mathematics, Louisiana State University, Baton Rouge, Louisiana 70817, USA} \par  
  N. Phanachet (corresponding author), \textit{E-mail address}: \texttt{natthawut.phanachet@gmail.com, nphana1@lsu.edu}
}}

\begin{document}

\maketitle

\begin{abstract}
An integral homology theory on the category of undirected reflexive graphs was constructed in \cite{BenayatTalbi2014}. A geometrical method to understand behaviors of $1$- and $2$-simplices under differential maps of the theory was developed in \cite{Phanachet2015} and led us to an independent proof that the first homology group of any cycle graphs is $\mathbb{Z}$, as it was proved before by a version of Hurewicz theorem harshly defined and shown in \cite{BenayatKadri1997} and \cite{BenayatTalbi2014}. In this work, we use the old method in \cite{Phanachet2015} to study behaviors of the first homology group of Hamiltonian graphs. We discovered that $H_1(G)$ is torsion-free, for any Hamiltonian graphs $G$. 
\end{abstract}

\section{\large Introduction}

In \cite{BenayatTalbi2014}, D. Benayat and M. E. Talbi constructed an integral homology theory on the category of undirected reflexive graphs in the same spirit as singular homology theory with the hope that the theory could be useful in dealing with problems in algebraic topology. The theory has main properties such as long exact sequences of pairs of graphs, the homotopy invariance and the excision theorem for suspension of graphs. To claim nontriviality of the theory, D. Benayat and M. E. Talbi defined a kind of graphs called graphical sphere $S^n$ and then showed that it has the same homology as the topological $n$-sphere. The $n$-sphere $S^n$ is built inductively from the suspension of $S^{n-1}$ as in the topological case, where $S^1$ is defined to be the cycle graph of $4$ vertices (possible nontrivial cycle graph with the smallest length). Besides from excision theorem in suspension version, the first homology group of $S^1$ is also needed to complete the homology calculation of any $S^n$. It is proved that $H_1(G)\cong\mathbb{Z}$ via a version of Hurewicz theorem shown in \cite{BenayatTalbi2014}, whereas a homotopy theory on the category of undirected reflexive graphs have been defined in \cite{BenayatKadri1997}.\\
\indent By our study on this homology theory, we set up our ideas in \cite{Phanachet2015} to reconsider $1$- and $2$-simplices in an easier way to understand them more geometrically under differential maps. This made us clear in their nature and consequently be able to guarantee that the first homology group of cycle graphs, especially $S^1$, is really $\mathbb{Z}$, satisfying D. Benayat and M. E. Talbi's result. This emphasized us that our approach to play with first homology is well enough.   We also showed in \cite{Phanachet2015} that any two cycle graphs with distinct number of vertices is non-homotopic to each other in the sense of homotopy defined in \cite{BenayatKadri1997}, \cite{BenayatTalbi2014}.\\
\indent In this paper, we use the method developed in \cite{Phanachet2015} to study the first homology group of a more general case of cycle graphs, Hamiltonian graphs, and what we mainly discovered in this study is that the first homology of any Hamiltonian graphs is always torsion-free. The idea behind the proof of this fact is very straightforward: we show that it is always possible to find a basis. \\
\indent Since the way we exhibit here only relies on \cite{BenayatTalbi2014}, \cite{Phanachet2015}, we intend to make this paper self-contained so that we collect all necessary things for readers in the next section. Note that we can imagine graphs in the category of undirected reflexive graphs, which we are going to work with as usual simple graphs. This makes sense because of \ref{class}. 

\section{\large Preliminaries}

\subsection{\normalsize Basic Definitions \& Notations}


\begin{defn} 
A \textbf{graph} $G$ is a pair $(V_G,\mathcal{N}_G)$ which $V_G$ is the set of vertices and 
$\mathcal{N}_G=\{\mathcal{N}_G(x):x\in V_G\}$, where $\mathcal{N}_G(x)$ is the set of neighbors of $x$.
Additionally, every vertex is a neighbor of itself: $x\in \mathcal{N}_G(x)$, $\forall x\in V$.
\end{defn}

\begin{defn}
A \textbf{morphism} $f:G\rightarrow G'$ is a function $f:V_G\rightarrow V_{G'}$ such that $f(\mathcal{N}_G(a))\subset \mathcal{N}_{G'}(f(a))$ for every $a\in V_G$.
\end{defn}

In this work, we adopt the categorical product, not the cartesian product, which is defined in the following definition.

\begin{defn}
If $G$ and $G'$ are graphs, their \textbf{categorical product}, written $G\times G'$, is defined by 
$(V_{G\times G'},\mathcal{N}_{G\times G'})$ where $\mathcal{N}_{G\times{G'}}(x,y)=\mathcal{N}_G(x)\times \mathcal{N}_{G'}(y)$ for $(x,y)\in V_G\times V_{G'}$.\\ Note that this product can be generalized to the finite product of graphs.
\end{defn}  

\begin{defn}
If $m\in \mathbb{N}\cup \{0\}$, the graph $\overline{\textbf{m}}$ is $(\{0,1,2,...,m\},\mathcal{N}_{\overline{m}})$ where 
$$\mathcal{N}_{\overline{m}}(a)=
\begin{cases}
	\{a-1,a,a+1\} &\text{if} \quad a\in \{1,2,...,m-1\};\\
	\{0,1\} &\text{if} \quad a=0;\\
	\{m-1,m\} &\text{if} \quad a=m.
\end{cases}
$$
We will use n-categorical product of $\overline{2}$, $\overline{2}^n$, which is easily written as $I_n$.
\end{defn}

\subsubsection{Singular n-simplices}

If $n\in \mathbb{N}$, a \textbf{singular n-simplex}  of the graph $G$ is a morphism $\sigma:I_n={\overline{2}}^n\rightarrow G$.\\  A $0$-simplex is just a morphism $\sigma:\{0\}\rightarrow G$ and can be identified with a vertex of 
$G$. Let us put $\textbf{S}_n(G)$ for the set of all $n$-simplices of $G$. If $R$ is a
unitary commutative ring, then the $R$-module $\mathcal{S}_n(G,R)$ of \textbf{\textit{n}-chains} of $G$ is the free $R$-module with the basis
$\textbf{S}_n(G)$: $$\mathcal{S}_n(G,R)=\{\sum_{\sigma\in \textbf{S}_n(G)} a_\sigma \cdot \sigma:a_\sigma\in R\text{ and are almost all zero}\}.$$
\indent An element of $\mathcal{S}_0(G,R)$ is a finite formal sum $\sum_{g\in V_G}a_g \cdot g$. From now on, we assume that $R=\mathbb{Z}$ and we omit mentioning it.\\ \indent If $n\geq 1$ and the index $j$ runs from $1$ to $n+1$, the index $k$ from $0$ to $1$, the \textbf{faces} of $I_{n+1}$ are the inclusions $f_{n+1}^{j,k} :I_n\rightarrow I_{n+1}$ defined by $$f_{n+1}^{j,k} (p_1,...,p_n)=(p_1,...,p_{j-1},2k,p_{j+1},...,p_n).$$
\indent When $n=0$, the faces of $I_{0+1}=\overline{2}$ are just the inclusions of the subgraphs $\{0\}$ and $\{2\}$ in $\overline{2}$. \\
\indent From above, we have the faces of the $(n+1)$-simplex $\sigma:I_{n+1}\rightarrow G$ are the composites $\sigma^{j,k}=\sigma\circ
f^{j,k}_{n+1}:I_n\rightarrow G$; so they are the restrictions of $\sigma$ to the faces of $I_n$ which really are $n$-simplices.\\
\indent We define the \textbf{differential} $\partial_{n+1}:\mathcal{S}_{n+1}(G)\rightarrow \mathcal{S}_n (G)$ on the generators $\sigma$ of $\mathcal{S}_{n+1}(G)$ as the sum $\partial_{n+1}(\sigma)=\sum_{j=1,...,n+1,k=0,1}(-1)^{j+k} \sigma^{j,k}$ and extend it by linearity. Then we have $\partial_n\circ \partial_{n+1}=0$ for all $n\geq 1$. So, $(\mathcal{S}_n(G),\partial_n)$ is a \textbf{chain complex}.

\subsubsection{Degenerate Singular Simplices}
The \textbf{degenerate faces} of $I_n$ are the $n$ projections $D_l:I_n\rightarrow I_{n-1}$ defined by $D_l(a_1,...,a_n)=(a_1,...,a_{l-1},a_{l+1},...,a_n)$ for $l=1,...,n$. For $n=0$, the degenerate face of $I_1=\overline{2}$ is the projection $D_1:I_1\rightarrow I_0=\{0\}$.\\
\indent The $n$-simplex $s:I_{n+1}\rightarrow G$ is \textbf{degenerate} if it factors through a degenerate face $D_l$ of
$I_{n+1}$; there is an $n$-simplex $\sigma$ such that $s=\sigma\circ D_l$. Let us put $\mathcal{D}_n(G)$ for the
$\mathbb{Z}$-submodule of $\mathcal{S}_n(G)$ generated by the degenerate $n$-simplices.\\
\indent It is clear that the quotient module $\mathcal{S}_n(G)/\mathcal{D}_n(G)$ is free with basis the cosets of nondegenerate $n$-simplices. It can be checked that $\partial_{n+1}(\mathcal{D}_{n+1}(G))\subset\mathcal{D}_n(G)$ for all $n\geq 1$. So, the differential $\partial_n$ can be factored through the quotients by the degenerate simplices, giving a differential $$\tilde{\partial}_n:\tilde{\mathcal{S}}_n(G)=\mathcal{S}_n(G)/\mathcal{D}_n(G)\rightarrow \mathcal{S}_{n-1}(G)/\mathcal{D}_{n-1}(G)=\tilde{\mathcal{S}}_{n-1}(G)$$ and providing us with a chain complex $(\tilde{\mathcal{S}}_n(G),\tilde{\partial}_n)$. As usual, we define the $n$-cycle, $n$-boundary and $n$th-homology group as, respectively:
\begin{align*}
Z_n(G)&=ker(\tilde{\partial}_n:\tilde{\mathcal{S}}_n\rightarrow\tilde{\mathcal{S}}_{n-1})\\
B_n(G)&=Im(\tilde{\partial}_{n+1}:\tilde{\mathcal{S}}_{n+1}\rightarrow\tilde{\mathcal{S}}_n)\\
H_n(G)&=Z_n(G)/B_n(G).
\end{align*} 

\subsubsection{Conventions}

\begin{enumerate}
		\item Every basis element of $S_1$ is denoted by the symbol $(abc)$, which $(abc)$ means a morphism of a graph 
	that sends $0$ to $a$, $1$ to $b$ and $2$ to $c$. We always imagine the picture of the basis element of $S_1$ with the parenthesis symbol.
		\item Every basis element of $S_2$ is denoted by the 3-square matrix $[a_{ij}]_{3\times 3}$ which means that
the $(r,s)$ vertex of $I_2$ is sent to $a_{(3-s)(r+1)}$ entry.     
		\item We will represent all elements of $H_1$ with $[X]$, where $X\in$ Ker$\partial_1$.
		\item As we have already defined, we let $\zeta$ to be the special letter occupied only for this situation: If we want to mention to the $S_1$ basis which we do intend to say about its $\bar{1}$ position(or the center position) in ways that the center is equal to the left or the right position, or is equal to both side in the case of equality. In this circumstance, we use $\zeta$ to represent the center. For instance, $(a\zeta b)$ can be interpreted to mean 
$(a\zeta b)=(aab)$ or $(a\zeta b)=(abb)$. 
		\item Every element of Ker$\partial_1$ is called \textbf{cycle} and say every cycle in Im$\partial_2$ is \textbf{trivial}. 
\end{enumerate}

\indent From our study on this homology theory, the following facts are keys to understand behaviors of all cycles.  

\subsubsection{Some Facts}
\begin{enumerate}
		\item The elements $(abc)-(abb)-(bcc)$, $(abc)-(aab)-(bbc)$ and $(abc)-(abb)-(bbc)$ are in 
		Im$\partial_2$. Because of the following consideration, we know how they can be obtained. There are
		$\left(\begin{smallmatrix} b&c&c \\ b&b&b \\ a&a&a \end{smallmatrix}\right)$, 
		$\left(\begin{smallmatrix} b&b&c \\ a&b&b \\ a&a&a \end{smallmatrix}\right)$,
		$\left(\begin{smallmatrix} b&b&c \\ b&b&b \\ a&a&a \end{smallmatrix}\right)$ $\in S_2$ 
		such that their values under $\partial_2$ being those elements respectively. 
		Moreover, $(abc)-(aab)-(bcc)$ is in Im$\partial_2$ because there is 
		$\left(\begin{smallmatrix} b&b&c \\ b&b&b \\ a&a&a \end{smallmatrix}\right)+$ 
		$\left(\begin{smallmatrix} b&b&b \\ a&b&b \\ a&a&a \end{smallmatrix}\right)+$
		$\left(\begin{smallmatrix} c&c&c \\ c&b&b \\ b&b&b \end{smallmatrix}\right)$.
		\item Similarly, due to $\left(\begin{smallmatrix} a&a&a \\ a&b&b \\ a&b&c \end{smallmatrix}\right)$,
		$\left(\begin{smallmatrix} a&a&a \\ a&a&a \\ a&b&a \end{smallmatrix}\right)$ and
		$\left(\begin{smallmatrix} a&b&b \\ a&a&b \\ a&a&b \end{smallmatrix}\right)$,
		the cycle $(abc)+(cba)$, $(aba)$ and $(aab)-(abb)$ are elements of Im$\partial_2$ respectively. 
		\item Given $\sum_{i=1}^n(abc)_i$ be an element in Ker$\partial_1$. By the fact 2, we attain \newline
		$[\sum_{i=1}^n(abc)_i]=-[\sum_{i=1}^n(cba)_i]$
		\item As mentioned above, the 3-square matrix is used to represent the basis element of $S_2$. If we imagine that 	matrix looks like a square, it can be roughly said that the class of the image $\partial_2$ of this square is invariant under rotating square. To have more understanding about this fact, let 
	$\left(\begin{smallmatrix} a&b&c \\ d&e&f \\ g&h&i \end{smallmatrix}\right)\in S_2$. We know that all of these bases 
	$$\left(\begin{smallmatrix} a&b&c \\ d&e&f \\ g&h&i \end{smallmatrix}\right),
	\left(\begin{smallmatrix} c&f&i \\ b&e&h \\ a&d&g \end{smallmatrix}\right),
	\left(\begin{smallmatrix} i&h&g \\ f&e&d \\ c&b&a \end{smallmatrix}\right),
	\left(\begin{smallmatrix} g&d&a \\ h&e&b \\ i&f&c \end{smallmatrix}\right)$$ are not all necessarily identical. 
	It can be checked that these are all possible forms of 
	$\left(\begin{smallmatrix} a&b&c \\ d&e&f \\ g&h&i \end{smallmatrix}\right)$ 
	under rotating clockwise or counterclockwise. \newline
	\end{enumerate}
	
\begin{defn}
A cycle in Ker$\partial_1$ is said to be a \textbf{proper cycle} if it is a \underline{summation} of basis elements which each of them has 1 as its coefficient.
\end{defn}

The underlined word means that this cycle is really summation, that is the minus sign does not exist. For an example, $(abc)+(cde)+(efg)+(gda)+(acd)+(dba)$, $(aba)$. Note that, in the rest of the paper, we always order the writing sequence of the bases of any proper cycle which is introduced later in the remarks unless it has an exception.

\begin{defn}
	A \textbf{shortest cycle} is a proper cycle consisting of only one basis element. 
\end{defn}

From previous example, $(aba)$. 

\begin{defn}
A proper cycle is said to be a \textbf{perfect cycle} if its each basis element consists of only two letters and the left position is not permitted to equal to the right. 
\end{defn}

For instance, $(aab)+(bcc)+(ccd)+(ddb)+(bba)$. 

\begin{defn}
\textbf{Simple cycle} is a perfect cycle with only two terms of $S_1$ basis elements. 
\end{defn}

For example, $(aab)+(bba)$. 

\begin{rem}\label{remark2}
1. Because each cycle is the combination of bases of $S_1$ which its image under $\partial_1$ is $0$, as the element of $H_1$, it can be rearranged the combination order of its bases which make it look like a real cycle. \newline 
\indent To give more details, If $A$ is a cycle. Consider $[A]$, by the above facts, we can transform $A$ into the other form called $A'$ which there is no minus sign and $[A]=[A']$. If $A'$ is the combination of $n$ terms, we let 
$\{f_n\}$ to be the corresponding morphisms and a new order for $A'$ we are looking for will follow this: $f_i(2)=f_{i+1}(0)$ for each $i\in \{1,...,n-1\}$ and $f_1(0)=f_n(2)$. \newline
\indent By this form, we can picture this cycle with a straight line combination beginning from left hand with each coefficient is only $1$, the first letter on the left equals to the last letter on the right and the right position of each basis except the last basis is same as the left position of the right one. \newline 
\indent For example, it is clear that $2(abc)-(bbc)+(baa)+(cba)\in$ Ker$\partial_1$. If we follow the previous form, we have $[2(abc)-(bbc)+(baa)+(cba)]=[(abc)+(cbb)+(baa)+(abc)+(cba)]$. \newline
\indent Therefore, we can conclude that every cycle always has a proper cycle which make them equal in the first homology group.\newline 
\indent 2. With the facts above combining with the definition of the proper cycle, we will obtain a result: every proper cycle always has a perfect cycle which is the same as the elements of the first homology group.\newline
\indent For example, there is a cycle $(aab)+(bbc)+(cbb)+(baa)+(abb)+(bcc)+(ccb)+(bba)$ which makes 
$[(abc)+(cbb)+(baa)+(abc)+(cba)]=[(aab)+(bbc)+(cbb)+(baa)+(abb)+(bcc)+(ccb)+(bba)]$.\newline
\indent Thus, it can be confirmed that every cycle always has a perfect cycle which make them identical in the first homology group.\newline
\indent Note that, from this remark, the shortest cycle can be regarded as a simple cycle in first homology group, that is, we have $(aab)+(baa)$ such that $[(aba)]=[(aab)+(baa)]$.
\end{rem}

\subsection{\normalsize Relation to Simple Graph Theory}

The following propositions are straightforward to see, however, they help us to work easier in section 3.  

\begin{prop}\label{class}
There is an one-to-one correspondence between the class of undirected reflexive graphs and the class of simple graphs. 
\end{prop}

\begin{prop}\label{set-isomorphic}
The set of all perfect cycles on a graph is set-isomorphic to the set of all cycles on the correspondent simple graph.  
\end{prop} 

\indent It now makes sense to mention to edges of a graph in our category.

\begin{defn}
Let $x,y$ be vertices of a graph $G$. If $x\in\mathcal{N}_G(y)$, we say that $x$ is \textbf{adjacent} to $y$ and $\{x,y\}$ the \textbf{edge} between $x$ and $y$ on $G$. We denote $E_G$ to be the set of all edges on $G$. 
\end{defn}

\indent Note that, for our convenience, summations of cycles of any graphs are considered to be as summations in the first homology group of that graphs unless we mention to. By \ref{set-isomorphic}, it makes sense to represent perfect cycles by cycles in the category of simple graphs. For instance, $(aab)+(bcc)+(caa)$ can simply write as $abca$. Also, by \ref{class}, it is reasonable to imagine interchangeably between undirected reflexive graphs and simple graphs, consequently, a notion such as connectedness and vertex deletion in simple graph theory are valid in our category as well.       

\section{\large First Homology Group of Hamiltonian Graphs} 

\begin{defn}
Let $G$ be a Hamiltonian graph with $n$ vertices and $a_1a_2...a_na_1$ be a Hamiltonian cycle. We relabel the names of all vertices of $G$ by setting $a_i$ to become $i$. An isomorphic new graph is called a \textbf{circle form} of $G$.   
\end{defn}

Roughly speaking, a Hamiltonian graph $G$ can be depicted as a cycle graph made by a Hamiltonian cycle and equipped with diagonal lines. See an example below.

\begin{center}
	\includegraphics[scale=0.83]{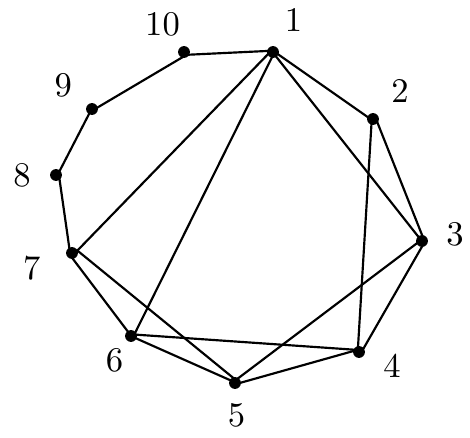}
\end{center}

It is not always true that a Hamiltonian graph has just only one Hamiltonian cycle. From on now, when we mention to a Hamiltonian graph $G$, it is understood that we are dealing with a fixed circle form of $G$ unless it has an exception.\\ 

Given $G$ a Hamiltonian graph with $n$ vertices. For each $v\in V_G$, we denote 
$$\overline{\mathcal{N}_G(v)}=
\begin{cases}
	\mathcal{N}_G(v)-\{v-1,v,v+1\} &\text{if}~~v=2,...,n-1;\\
	\mathcal{N}_G(v)-\{1,2,n\} &\text{if}~~v=1;\\
	\mathcal{N}_G(v)-\{1,n-1,n\} &\text{if}~~v=n.
\end{cases}	
$$

Note that if $G$ is a cycle graph, we apparently have $\overline{\mathcal{N}_G(v)}=\emptyset$ for all $v$. \\

For a Hamiltonian graph $G$ which is not a cycle, if $u\in \overline{\mathcal{N}_G(v)}$, we will say that $\mathbi{u~and~v~are}$ $\mathbi{diagonally~adjacent}$. Then, we create a set 
$$\mathcal{O}_G=\{\{r,s\}:r,s~ \text{are~diagonally~adjacent}\}$$ and 4 types of perfect cycles associating to each edge $\{r,s\}$ of $\mathcal{O}_G$: if $r<s$,
\begin{itemize}
\item Type 1. $rs(s+1)...n12...(r-1)r$;
\item Type 2. $rs(s-1)...(r+1)r$;
\item Type 3. $sr(r-1)...1n(n-1)...(s+1)s$; 
\item Type 4. $sr(r+1)...(s-1)s$.
\end{itemize}

For each edge $\{r,s\}$, we let $A_{\{r,s\}}$ to be a set which consists of homology classes of all 4 types of cycles corresponding to the edge $\{r,s\}$.

\begin{lem}\label{lemma}
The set $\mathfrak{C}=\{[12...(n-1)n1]\}\cup\Big{(}\bigcup_{\{r,s\}\in \mathcal{O}_G} A_{\{r,s\}}\Big{)}$ spans $H_1(G)$.
\end{lem}
\begin{proof}
	Let $[A]$ be a nontrivial element of $H_1(G)$ where $A=a_1a_2...a_pa_1$ and $a_2,...,a_p$ being all distinct by the sake of \ref{remark2}. \\
\indent If there is no element of $\mathcal{O}_G$ on $A$, it is obvious to see that $A$ is $12...(n-1)n1$ and the proof is completed. \\
\indent If there is an element of $\{a_ja_{j+1}, a_pa_1: 1\leq j < p\}\subset \mathcal{O}_G$ on $A$, we immediately know that $A$ is one of $4$ types and then finish the proof. \\
\indent If there are only two elements of $\{a_ja_{j+1}, a_ka_{k+1},a_pa_1 : 1\leq j,k < p, j\not=k\}\subset \mathcal{O}_G$ on $A$, we have 
\begin{align*}
[a_1a_2...a_ja_{j+1}...a_ka_{k+1}...a_pa_1]&=[a_ja_{j+1}...a_ka_{k+1}...a_pa_1...a_j]\\
&=[a_ja_{j+1}...a_{k-1}*a_j]+[a_j*'a_{k-1}a_ka_{k+1}...a_pa_1...a_j],
\end{align*}
where $*$ is the sequence of distinct vertices which 1.) are also different from $a_{j+1},...,a_{k-1}$ and 2.) makes $a_ja_{j+1}...a_{k-1}*a_j$ have no element of $\mathcal{O}_G$ except $a_ja_{j+1}$. Conversely, $*'$ is just the backward sequence of $*$. It is clear that $[a_ja_{j+1}...a_{k-1}*a_j]\in \mathfrak{C}$ but 
$[a_j*'a_{k-1}a_ka_{k+1}...a_pa_1...a_j]$ might not be straightforward like the former. However, the latter can be simplified into a form of $\mathfrak{C}$ by reducing the representative to a perfect cycle. 
\\
\indent For the rest of the cases, we apply the consideration of the previous case to them. Therefore, the only thing we have to do is repeating the step until there is no homology class which its representative contains more than one element of $\mathcal{O}_G$. 
\end{proof}

\begin{rem}\label{remark3}
\begin{enumerate}
	\item From the proof, we are able to see that every possible combination of a perfect cycle has just only $1$ as the coefficients.
	\item For each element of $\mathcal{O}_G$, it is not difficult to see that: \newline
	 	 (i) $~~$[cycle in Type 2] = $-$[Hamiltonian cycle] + [cycle in Type 1]\newline
	 	 (ii)  $~$[cycle in Type 3] = $-$[cycle in Type 1]\newline
	 	 (iii) [cycle in Type 4] = $-$[cycle in Type 2] = [Hamiltonian cycle] $-$ [cycle in Type 1]	   
\end{enumerate} 
\end{rem}
 
By \ref{remark3}(2), we immediately know that $\mathfrak{C}$ is not linearly independent. Hence, to attain a basis of $H_1(G)$ from $\mathfrak{C}$, we have to remove some its elements off. \\

In the next proposition, we generalize the idea of triviality shown in Lemma 3.6 of \cite{Phanachet2015} to apply for any graphs. For our convenience, we abbreviate the word ``perfect cycle of length 3" as only just ``3-perfect cycle". 

\begin{thm}\label{trivial}
Let $A$ be an element of Ker$\partial_1$ of a graph $G$. There is a combination of $3$-perfect cycles $B$ such that 
$[A]=[B]$ if and only if $A\in$ Im$\partial_2(G)$.
\end{thm}

\begin{proof}
For a $3$-perfect cycle of $G$, say $(a\ze b)+(b\ze c)+(c\ze a)$, there always have a $2$-simplex 
$\left(\begin{smallmatrix} a&c&c \\ a&b&b \\ a&a&b \end{smallmatrix}\right)$ which its image under $\partial_2$ is that cycle. Thus, $[B]$ is trivial and then $A$ is automatically in Im$\partial_2(G)$.\\
\indent Conversely, if $A\in$ Im$\partial_2(G)$, there is a combination of basis elements of $S_2(G)$ which its image is $A$. To complete the proof, it suffices to show that each of these basis elements has the image as a combination of simple cycles or $3$-perfect cycles. Let $\alpha=\left(\begin{smallmatrix} a&b&c \\ d&e&f \\ g&h&i \end{smallmatrix}\right)$
be an element of $\mathcal{S}_2(G)$ with assuming that $a,b,...,i$ are all distinct. We have 
$$\partial_2(\alpha)=(ghi)+(ifc)-(abc)-(gda).$$
However, 
\begin{align*}
[(ghi)+(ifc)-(abc)-(gda)]&= [(ghi)+(ifc)+(cba)+(adg)]\\
&= [(g\ze h)+(h\ze i)+(i\ze f)+(f\ze c)+(c\ze b)\\
&~~~+(b\ze a)+(a\ze d)+(d\ze g)]\\
&=[((g\ze h)+(h\ze e)+(e\ze g))+((h\ze i)+(i\ze e)+(e\ze h))\\
&~~~+...+((d\ze g)+(g\ze e)+(e\ze d))]\\
&=[gheg+hieh+ifei+fcef+cbec\\
&~~~+baeb+adea+dged].
\end{align*}
Hence, there is a combination of $3$-perfect cycles $B=gheg+hieh+ifei+fcef+cbec+baeb+adea+dged$ such that $[A]=[B]$.
\end{proof}

\begin{defn}
Let $A=a_1a_2...a_ma_1$ be a perfect cycle on a graph $G$ with $m>3$. We call $A$ a \textbf{completely perfect cycle}, if there is no pair $i,j\in\{1,...,m\}$ such that $a_i$ is adjacent to $a_j$ on $G$ whenever $i+1<j$. 
\end{defn}

\begin{rem}
\indent 1. We have seen from the method in proving Lemma 3.2 and Theorem 3.4 that any two perfect cycles that have an intersection on a path of a graph can produce the new cycle getting from combining these two cycles. We will call this kind of cycle summation as \textbf{intersection-free summation}. Note that there are only two possible new orientations of cycles arisen from the intersection-free summation of perfect cycles. For our convenience in writing, we will not mention to which orientation we are working with since the readers can notice from the context. \\ 
\indent 2. For any perfect cycle $a_1a_2...a_ma_1$ of a graph $G$ with $m>3$, if there are $i,j\in\{1,...,m\}$ such that $i+1<j$ and $a_i$ is adjacent to $a_j$ on $G$, we can split the cycle into the intersection-free summation of two perfect cycles, say $a_1...a_ia_ja_{j+1}...a_ma_1$ and $a_ia_{i+1}...a_ja_i$. The edge $a_ia_j$ is called a \textbf{splitting edge} of $a_1a_2...a_ma_1$ on $G$. With this fact, a perfect cycle $A$ on a graph can be written as a summation of completely perfect cycles and 3-perfect cycles on $G$. We call these cycles as \textbf{cycle components} of $A$.\\
\indent 3. We will use the word \textbf{triangular graphs} instead when dealing with cycle graphs of length 3. Also, the symbol $\D abc$ denotes the triangular subgraph having the vertices $a,b$ and $c$. 
\end{rem}

\begin{thm}\label{hardest}
Let $A$ be a trivial perfect cycle on a graph $G$. There always is a connected subgraph $H$ of $G$ which contains $A$ and every cycle on $H$ is trivial. 
\end{thm}

\begin{proof}
Let assume that $A=a_1a_2...a_ma_1$. It is trivial for the cases where $m\leq3$. So, suppose that $m>3$. Since $A$ is trivial, by theorem 3.4, there are $3$-perfect cycles on $G$ which their summation is $A$. So, the union of all triangular subgraphs corresponding to these $3$-perfect cycles is a subgraph of $G$ containing $A$, say $G'$. Also, for any edge $xy$ of any triangular subgraph of $G'$, if $xy\notin\{a_1a_2,a_2a_3,...,a_ma_1\}$, there is another triangular subgraph of $G'$ having $xy$ as its edge. \\
\indent We first assume that $A$ is completely perfect. We construct a subgraph $H$ of $G'$ as follows: \\
\indent 1. Since $G'$ contains $A$, we pick a triangular subgraph of $G'$ which has the $a_1a_2$ edge, say $\D a_1a_2\al^{1,2}_0$ and another one which has the $a_2a_3$ edge, say $\D a_2a_3\al^{2,3}_0$. Note that $\D a_1a_2\al^{1,2}_0\not=\D a_2a_3\al^{2,3}_0$.  \\
\indent 2. Since $A$ is completely perfect, we have $a_2\al^{1,2}_0,a_2\al^{2,3}_0\notin\{a_1a_2,a_2a_3,...,a_ma_1\}$. \\
\underline{Case1} $\al^{1,2}_0=\al^{2,3}_0$. So, $a_2\al^{1,2}_0=a_2\al^{2,3}_0$. We construct the subgraph $H^{1,2,3}$ of $G'$ by union $\D a_1a_2\al^{1,2}_0$ and $\D a_2a_3\al^{2,3}_0$. Thus, the intersection-free summation of two 3-perfect cycles corresponding to these triangular subgraphs is $a_1a_2a_3\al^{1,2}_0a_1$.\\
\underline{Case2} $\al^{1,2}_0\not=\al^{2,3}_0$. So, $a_2\al^{1,2}_0\not=a_2\al^{2,3}_0$. By the fact above, there are $\D a_2\al^{1,2}_0\al^{1,2}_1$ and $\D a_2\al^{2,3}_0\al^{2,3}_1$ of $G'$.  \\
\indent \underline{2.1} $\D a_2\al^{1,2}_0\al^{1,2}_1=\D a_2\al^{2,3}_0\al^{2,3}_1$. The subgraph $H^{1,2,3}$ is the union of  $\D a_1a_2\al^{1,2}_0$, $\D a_2a_3\al^{2,3}_0$ and $\D a_2\al^{1,2}_0\al^{1,2}_1$. The intersection-free summation of three 3-perfect cycles corresponding to these triangular subgraphs is $a_1a_2a_3\al^{2,3}_0\al^{1,2}_0a_1$.\\
\indent \underline{2.2} $\D a_2\al^{1,2}_0\al^{1,2}_1\not=\D a_2\al^{2,3}_0\al^{2,3}_1$. We apply Step 1.--2.1 to the vertices $\al^{1,2}_1,\al^{2,3}_1$ of $\D a_2\al^{1,2}_0\al^{1,2}_1$, $\D a_2\al^{2,3}_0\al^{2,3}_1$, respectively. By the finiteness of the graph $G'$, the construction must finally end and, by relabeling the names of vertices, the perfect cycle corresponding the intersection-free summation of 3-perfect cycles on $H^{1,2,3}$ is $a_1a_2a_3\al^{2,3}\al^{1,2}_{k_1}...\al^{1,2}_1\al^{1,2}_0a_1$ for some $k_1\in\mathbb{N}$.  \\
\indent 3. Apply Step 1. and 2. to the pair $a_2a_3,a_3a_4$ but fix the triangular subgraph $\D a_2a_3\al^{2,3}_0$. We then have the perfect cycle corresponding the intersection-free summation of 3-perfect cycles on $H^{2,3,4}$ is $a_2a_3a_4\al^{3,4}_0\al^{2,3}_{k_2}...\al^{2,3}_1\al^{2,3}_0a_2$ for some $k_2\in\mathbb{N}$. \\
\indent 4. Continue applying Step 1., 2. and 3. to the rest of edge pairs until the last pair $a_ma_1,a_1a_2$, we skip Step 1. and fix both $\D a_ma_1\al^{m,1}_0,\D a_1a_2\al^{1,2}_0$ which attain from the previous construction and the perfect cycle corresponding the intersection-free summation of 3-perfect cycles on $H^{m,1,2}$ is $a_ma_1a_2\al^{1,2}_0\al^{m,1}_{k_m}...\al^{m,1}_1\al^{m,1}_0a_m$ for some $k_m\in\mathbb{N}$. We define $H_1:=\bigcup^m_{i=1}H^{i,i+1,i+2}$ and call it the $\mathbi{layer-1~subgraph~of~G'}$.
\\
\indent 5. We have the intersection-free summation of all 3-perfect cycles corresponding to all triangular subgraphs of $H_1$ is $$a_1a_2...a_ma_1+\al^{1,2}_0\al^{m,1}_{k_m}...\al^{m,1}_1\al^{m,1}_0\al^{m-1,m}_{k_{m-1}}...\al^{m-1,m}_1\al^{m-1,m}_0...\al^{2,3}_0\al^{1,2}_{k_1}...\al^{1,2}_1\al^{1,2}_0.$$ Let say the latter cycle the $\mathbi{induced~cycle~of~A~on~H_1}$ or, shortly, $A_1$. Rename the vertices of $A_1$ respectively by letting $\al^{1,2}_0:=b_1,\al^{m,1}_{k_m}:=b_2,...,\al^{1,2}_2:=b_{n-1},\al^{1,2}_1:=b_n$. An edge $xy$ on $H_1$ is called a 
$\mathbi{path~edge}$ on $H_1$ if $xy$ is the edge which $A$ or $A_1$ goes through. Otherwise, $xy$ is called a $\mathbi{level~edge}$ of $A$ and $A_1$ on $H_1$.\\
\indent 6. Since $A_1$ is perfect, by Remark 3.6(2), $A_1$ is a summation of cycle component. Let $\mathcal{C}(A_1)$ be the set of all splitting edges of $A_1$ on $G'$. We will split the graph $H_1\cup\mathcal{C}(A_1)$ into a union of its subgraphs respecting to $\mathcal{C}(A_1)$ as follows: pick $b_sb_t\in\mathcal{C}(A_1)$  where $s<t$. So, there are level edges $a_jb_t,a_kb_s$ of $H_1$ such that are paths connecting $A$ and $A_1$. Then, we split $H_1\cup\mathcal{C}(A_1)$ as a union of its subgraphs which are a subgraph induced by the vertices $a_j,a_{j+1},...,a_k,b_s,b_{s+1},...,b_t$ and a subgraph induced by $a_j,a_{j-1},...,a_1,a_m,...,a_{k+1},a_k,b_s,b_{s-1}$, $...,b_1,b_n,...,b_{t+1},b_t$. Roughly speaking, $H_1$ is chopped into two pieces along the path $a_kb_sb_ta_j$. Continue this process to both subgraphs with other elements of $\mathcal{C}(A_1)$ and so on. Since $\mathcal{C}(A_1)$ is finite, we finally have a collection of subgraphs of $H_1\cup\mathcal{C}(A_1)$ and we call each subgraph a $\mathbi{layer-1~subgraph~of~G'~respecting~to~A_1}$, denoted by $L_1$. Note that each layer-1 subgraph respecting to $A_1$ contains only one of cycle components of $A_1$.  \\
\indent 7. For each $L_1$ containing a completely perfect cycle, we apply Step 1.-5. to that cycle and then have the layer-2 subgraph $H_2$ and the induced perfect cycle $A_2$. By construction, for each vertex $x$ of $A_2$, there always is a path $xyz$ connecting $x$ of $A_2$ and $z$ of $A$ which $xy$ and $yz$ are level edges of $H_2$ and $H_1$, respectively. Then, we split $L_1\cup H_2$ into a union of its subgraphs respecting to $\mathcal{C}(A_2)$ in the same manner of Step 6 with using paths described previously for splitting. As Step 6., we call each subgraph, obtained from the splitting, a layer-2 subgraph of $G'$ respecting to $A_2$, denoted by $L_2$. For each $L_1$ containing a 3-perfect cycle, we stop the construction process.       \\
\indent 8. For each $L_2$, we continue applying Step 1.-5. to $A_2$ and then splits $L_2\cup H_3$ as Step 6.-7. to get $L_3$'s. Thus, we continue the construction process by repeating this manner to $L_3$'s with their induced cycles $A_3$ and so forth. However, by the finiteness of $G'$ and the fact that, for any edge $xy$ such that $xy\notin\{a_1a_2,a_2a_3,...,a_ma_1\}$, there is a triangular subgraph of $G'$ having $xy$ as its edge, the construction process must end in finite repetition,   we For each cycle component subgraph respecting to $H_1$ and $A_1$ which the cycle component is  completely perfect, apply Step 1.--4. to the cycle component and call a new subgraph of $G'$ and a new cycle on it as $H_2$ and $A_2$, respectively. By applying Step 5.--7. to $H_2$, we have a collection of cycle component subgraphs respecting to $H_2$ and $A_2$. So, let $H$ be the union of: 1.) all cycle component subgraphs respecting to $H_2$ and $A_2$ of all cycle component subgraphs respecting to $H_1$ and $A_1$ which their cycle component are completely perfect and 2.) all cycle component subgraphs respecting to $H_1$ and $A_1$ which their cycle component are 3-perfect.  \\
\indent 9. Continue applying Step 1.--8. to all $A_2$'s and then get a new $H$. When applying Step 1.--8. to get new prefect cycles and a new $H$, we say that it is a $\mathbi{round}$ of construction process. However, by the finiteness of $G'$ and the fact that, for any edge $xy$ such that $xy\notin\{a_1a_2,a_2a_3,...,a_ma_1\}$, there is a triangular subgraph of $G'$ having $xy$ as its edge, the construction process must end, that is, new prefect cycle $A_N$'s at the $N$th round for some $N\in\mathbb{N}$ are all trivial. Therefore, the graph $H$ of the $N$th round is our desired subgraph of $G'$. \\
\indent 10. Apply Step 1.--5. to $A_1$ and attain $H_2$ and $A_2$. So, let the graph $H$ is the union of $H_1$ and $H_2$. If $A_2$ is completely perfect, repeat the process by applying Step 1.--5. to $A_2$ and so forth until there is a cycle $A_k$ for some $k\in\mathbb{N}$ which is not completely perfect. Then, apply Step 6.--
\indent Note that, by the construction, for each trivial $A_i$, where $1\leq i\leq N$, there is only one cycle component subgraph respecting to $H_i$ and $A_i$, say $L_i$. If $i>1$, there is only one cycle component subgraph respecting to $H_{i-1}$ and $A_{i-1}$ which produces $L_i$, say $L_{i-1}$. With this reason, we can inductively collect subgraphs of $H$ which are $L_1,...,L_i$. The subgraph $L:=\bigcup_{\xi=1}^i L_\xi$ is called a $\mathbi{layered~subgraph}$ of $H$. By the construction, the intersection of each consecutive pair $L_i,L_{i+1}$ is a collection of paths which means that the intersection is nonempty. Also, each $L_i$ is connected. Hence, any layered subgraph $L$ of $H$ is always connected.   \\
\indent For the general case of $A$, by Remark 3.6(2), $A$ can be written as a summation of cycle components. So, we apply the previous construction to each cycle component of $A$ and then have a collection of subgraphs of $G'$. Thus, the union of all those subgraphs is the desired subgraph $H$ of $G'$ which automatically is a subgraph of $G$. Equivalently, by the note above, $H$ is the union of all its layered subgraphs. Note that each splitting edge of $A$ is also called a path edge.\\
\indent By the construction of $H$, the intersection-free summation of all 3-perfect cycles corresponding to all triangular subgraphs on $H$ is $A$. Thus, $H$ is connected since if it is not, the intersection-free summation gives $A$ as a summation of two distinct perfect cycles. \\
\indent Let $C=c_1...c_pc_1$ be a perfect cycle on $H$. If $c_1...c_q$ is the longest path of $C$ on a layered subgraph $L$, where $1<q<p$, by connectedness of $L$, there is a path from $c_1$ to $c_q$ on $L$, say $c_1d_1d_2...d_uc_q$. So, $C$ can be separated into a summation of a cycle on $L$ and a cycle on $H$ outside $L$, that is, $C=c_1...c_qd_u...d_1c_1+c_1d_1...d_uc_qc_{q+1}...c_pc_1$. If $p=q$, $C$ must be a cycle on a layered subgraph and we call $C$ is a $\mathbi{layered~cycle}$. With this fact, any perfect cycle on $H$ can be written as a summation of distinct layered cycles.   \\
\indent Therefore, to show that $C$ is trivial, it suffices to show for the case of $C$ being a layered cycle on a layered subgraph $L$. Let $L=\bigcup_{\xi=1}^n L_\xi$ for some $n\in\mathbb{N}$. For the convenience in explaining, we let $A_0:=A$ and $m$ is the smallest number which $C$ travels through $A_\xi$ for $\xi\in\{0,1,...,n\}$. By recycling the order of vertices of $C$, let $c_1...c_v$ be a path of $C$ on $A_m$ which $c_pc_1$ and $c_vc_{v+1}$ are level edges of $A_m$ and $A_{m+1}$. Since $A_{m+1}$ is a cycle, there is a path $c_pe_1...e_wc_{v+1}$ of $A_{m+1}$ such that $C$ can be written as $c_1...c_vc_{v+1}e_we_{w-1}...e_1c_pc_1+c_pe_1...e_wc_{v+1}c_{v+2}...c_p$. By       \\

\indent It is clear that $H$ contains $A$ and meets the first property. To show that $H$ is connected, assume in the contrary that $H$ is not. So, $H$ must have more than one component. This means that $A$ is not able to be just only one perfect cycle. 
\end{proof}

From now on, if we mention to the $1$-simplex corresponding to an edge $\{r,s\}$ and not regarding to an order of $r$ and $s$, it is denoted by $(r\ze s)'$. 

\begin{defn}
Let $\{p,q\}$ and $\{r,s\}$ be edges of $\mathcal{O}_G$. Whenever there is a trivial perfect cycle containing 
$(p\ze q)'$ and $(r\ze s)'$, we say $\{p,q\}$ and $\{r,s\}$ are \textbf{edge-connected}. 
\end{defn}

\begin{defn}[Nets]
Let $\mathcal{G}$ be a subset of $\mathcal{O}_G$. If $\mathcal{G}$ holds a property that any pair of its edges is always edge-connected then we say that $\mathcal{G}$ is an \textbf{edge-connected set}. Furthermore, if $\mathcal{G}$ is also a maximal subset, $\mathcal{G}$ is called a \textbf{net} of $G$.  
\end{defn}

\begin{rem}
1. the edge-connectedness does not satisfy the transitive property. A graph below is used to show our claim.

\begin{center}
	\includegraphics[scale=0.83]{Fig1}
\end{center}

By the graph above, it can be seen that $13576421$ and $16571$ are trivial perfect cycles since $$[13576421]=[1321]+[3543]+[5765]+[4564]+[2342]~~~\text{and}$$
$$[16571]=[1671]+[5765].$$

Therefore, $\{1,3\},\{5,7\}$ and $\{5,7\},\{1,6\}$ are edge-connected. Conversely, $\{1,3\}$ and $\{1,6\}$ are not edge-connected since it is impossible to have a trivial perfect cycle on this graph traveling through $\{1,3\}$ and $\{1,6\}$. This means that $\{1,3\}$ is not in the same net containing $\{1,6\}$ and then there must have at least two nets on the graph.

2. From the first remark, we imply that the edge-connectedness cannot be an equivalence relation. That is, any two nets might intersect vacuously or not. Practically, it can be checked that this graph has just only two nets and they intersect to each other. That is,
\begin{align*}
\mathcal{G}_1&=\{\{1,3\},\{2,4\},\{3,5\},\{4,6\},\{5,7\}\} \\
\mathcal{G}_2&=\{\{1,6\},\{1,7\},\{2,4\},\{3,5\},\{4,6\},\{5,7\}\}.
\end{align*}
\end{rem}

Let $\mathcal{G}$ be an edge-connected set on a Hamiltonian graph $G$ and consider a unique subgraph $s(\mathcal{G})$ of $G$ defined by:
\begin{enumerate}
\item $x\in V_{s(\mathcal{G})}$ iff there are edges $\{p,q\},\{r,s\}\in\mathcal{G}$ such that $x$ is on a trivial perfect cycle traveling through $(p\ze q)',(r\ze s)'$.
\item $y\in \mathcal{N}_{s(\mathcal{G})}(x)$ iff $y\in \mathcal{N}_{G}(x)$ and there are edges $\{p,q\},\{r,s\}\in\mathcal{G}$ such that either $(x\ze y)'$ is on a trivial perfect cycle traveling through $\{p,q\}',\{r,s\}'$.
\end{enumerate}
We weigh all edges of $s(\mathcal{G})$ by defining a weight $w:E_{s(\mathcal{G})}\rightarrow \mathbb{N}\cup\{0\}$ as follows: 
$$w(\{r,s\}) = \begin{cases}
				0 &\mbox{if}~\{r,s\}\notin\mathcal{O}_G;\\
				1 &\mbox{if}~\{r,s\}\in\mathcal{O}_G.
				\end{cases}$$
Applying the $\mathbi{minimum~spanning~tree~algorithm}$ to the weighted graph of $s(\mathcal{G})$ and then getting a minimum spanning subtree $T$ of $s(\mathcal{G})$. We call the set $\mathcal{T}_{T,\mathcal{G}}:=E_T\cap\mathcal{G}$ a \textbf{spanning set} of the edge-connected set $\mathcal{G}$ based on $T$.

\begin{rem}\label{remark net}
1. Since every edge-connected set $\mathcal{G}$ is contained in a net $\mathcal{G}'$, we have that $s(\mathcal{G})$ must be a subgraph of $s(\mathcal{G}')$.   \\
\indent 2. The algorithm does not guarantee the uniqueness of a minimum spanning subtree $T$ of $s(\mathcal{G})$, so is 
$\mathcal{T}_{T,\mathcal{G}}$. \\
\indent 3. All edges of $s(\mathcal{G})$ outside $\mathcal{O}_G$ are always selected by the minimum spanning tree algorithm. Thus, we know that the difference between any two minimum spanning subtrees of $\mathcal{G}$ is the chosen edges from $\mathcal{O}_G$.  \\
\indent 4. We know that the number of edges of a spanning subtree of a graph is the number of the graph vertices minus $1$. By 2., we immediately know that the order of a spanning set $\mathcal{T}_{T,\mathcal{G}}$ is unique for every $T$ and $|\mathcal{T}_{T,\mathcal{G}}|$ can be computed by the formula:
$$|\mathcal{T}_{T,\mathcal{G}}|=|V_{s(\mathcal{G})}|-|E_{s(\mathcal{G})}\cap(E_G-\mathcal{O}_G)|-1.$$ 
Clearly, $\mathcal{T}_{T,\mathcal{G}}$ might be vacuous and then $|\mathcal{T}_{T,\mathcal{G}}|=0$.
\end{rem}

From now on, if we want to mention to a spanning set of $\mathcal{G}$ not regarding to its minimum spanning subtree $T$, we will denote $\mathcal{T}_{T,\mathcal{G}}$ as $\mathcal{T}_\mathcal{G}$, or $\mathcal{T}$ if it is unambiguous by the context.







\begin{thm}\label{key theorem}
Let $A$ be a perfect cycle on a Hamiltonian graph $G$. Then $A$ is a cycle on $s(\mathcal{G})$ of a net $\mathcal{G}$ if and only if $A\in$ Im$\partial_2$ of $G$.
\end{thm}  

\begin{proof}
Let $A=a_1a_2...a_na_1$ on $s(\mathcal{G})$ of a net $\mathcal{G}$. If $n=3$, it is apparently true. So, assume that $n>3$. We consider the subgraph of $s(\mathcal{G})$ obtaining by deleting a vertex $x$ on $s(\mathcal{G})$ which holds the following properties: $1.)$ $x$ is not on $A$ and $2.)$ there is an adjacent edge of $x$ being not a co-edge. Repeat the same process to $s(\mathcal{G})-x$ if it is necessary until the final subgraph has each vertex appears on $A$ or, if not, its adjacent edges are all co-edges. We name this subgraph $K$. Next, choose a triangular subgraph, say $\D a_1a_2\alpha$, having $\{a_1,a_2\}$ as its edge. Since $n>3$, it needs more than a triangular subgraph for $A$ to travel on. Thus, there is an edge of $\D a_1a_2\alpha$ which is a co-edge. WLOG, assume that $\al\not= a_3$. So, $\{a_2,\al\}$ and $\{a_1,\al\}$ are co-edges. By the property of $K$ and finiteness of vertices, there is a sequence of vertices $\al,\la_1,...,\la_k$ on $\mathcal{G}''$ which associates to triangular subgraphs $\D a_2\al\la_1,\D a_2\la_1\la_2,...,\D a_2\la_{k-1}\la_k,\D a_2\la_ka_3$ having $\{a_2,\la_1\},\{a_2,\la_2\},...,\{a_2,\la_k\}$ as all co-edges. With the same process, we continue collecting the triangular subgraphs and stop at the case of $\{a_{n-1},a_n\}$ and suppose that the triangular subgraph containing $\{a_{n-1},a_n\}$ is $\D a_{n-1}\gamma a_n$. For the case of $\{a_n,a_1\}$, it is possible that there is no sequences of vertices, beginning with $\gamma$ and ending with $\al$, having the property as previous ones. However, there must have a sequence $\gamma,...,\beta$ where $\beta$ is a vertex of $\D a_n\beta a_1$ and $\beta\not=\al$. Therefore, we finally obtain a collection of triangular subgraphs such that when combining them we get a subgraph of $K$, say $K_1$. We do a summation of $3$-prefect cycles corresponding to all triangular subgraphs we get previously by letting the direction of summation following the direction of $A$. What we get is the cycle 
$$a_1a_2...a_na_1a_2\beta...\gamma...\la_k\la_{k-1}...\la_1\al a_1,$$ which can be split into two perfect cycles, that are $a_1a_2...a_na_1$ and $a_1a_2\beta...\gamma...\la_k\la_{k-1}...$ \\ $\la_1\al a_1$. It is evident to see that we can apply the same method above to the subgraph $K_2$ of $K$ obtaining from deleting $a_3,...,a_n$ out of $K$ and the cycle $a_1a_2\beta...\gamma...\la_k\la_{k-1}...\la_1$ \\ $\al a_1$, since $K_2$ also have the same properties as $K$. So, we again get a new subgraph of $K$ and a new related perfect cycle. By the finiteness of vertices, the new cycle getting must be shorter than the former. Consequently, the process of selecting triangular subgraphs from $s(\mathcal{G})$ is finished by choosing the triangular subgraph corresponding to the last $3$-perfect cycle and then clearly see that the summation of all $3$-perfect cycles getting from above process in the appropriate direction is $A$. This concludes that $A\in$ Im$\partial_2(G)$ by \ref{trivial}.\\
\indent Conversely, if $A\in$ Im$\partial_2(G)$, by \ref{hardest}, there is a connected subgraph $K$ of $G$ contains $A$ and every cycle on $K$ is trivial. So, $E_K\cap\mathcal{O}_G$ is an edge-connected set, therefore, $K$ is clearly a subgraph of $s(E_K\cap\mathcal{O}_G)$. Hence, by \ref{remark net}(1), $K$ must be a subgraph of $s(\mathcal{G}')$, for some net $\mathcal{G}'$. 
\end{proof}

\begin{rem} Given a net $\mathcal{G}$ of a graph $G$ with a fixed circle form $\mathfrak{F}$. It is clear that some edges of $s(\mathcal{G})$ might be parts of $\mathfrak{F}$. By applying minimal spanning tree method, we are able to obtain a spanning subtree of $s(\mathcal{G})$ which those parts are all collected. This can be easily done by assuming all edges of $s(\mathcal{G})$ on $\mathfrak{F}$ are weighted by $0$ and the others are $1$. 
\end{rem}

\begin{defn}
Given $\mathcal{T}$ be a spanning set of a net $\mathcal{G}$. The subgraph $\mathcal{T}(\mathfrak{B})$ obtained by deleting all parts of $\mathcal{T}$ on $\mathfrak{F}$ is called \textbf{basis subgraph} corresponding to the net $\mathcal{G}$.   
\end{defn}

\begin{rem}\label{remark1}
1. We denote $\mathcal{T}(\mathfrak{B})^i$ to be the subset of $\bigcup A_{\{r,s\}}$ with all nontrivial classes in Type $i$ cycles on $G$.\\
\indent 2. For our convenience, whenever we are not concerned with the type (in 4 types) of a cycle corresponding to $\{r,s\}\in \mathcal{O}_G$, the symbol $[(r\ze s)]$ can be used to represent the homology class of all 4 types of cycles. Furthermore, the symbol $[(r\ze s)]^i$ means the homology class of Type $i$ cycle.\\
\indent 3. By \ref{key theorem}, it is possible that the homology class of Hamiltonian cycle or any element of $\mathcal{T}(\mathfrak{B})^i$ is trivial.
\end{rem}

\begin{thm}\label{key theorem 2}
Let $G$ be a Hamiltonian graph with a fixed circle form. If $G$ has just only one net $\mathcal{G}$ with its spanning set $\mathcal{T}$, the homology class of Hamiltonian cycle and $\mathcal{T}(\mathfrak{B})^i$ form a basis of $H_1(G)$, for any $i$.  
\end{thm}

By \ref{remark1}(4), the theorem still makes sense even though $\mathcal{T}(\mathfrak{B})^i$ is empty. If it happens, we can conclude that $H_1(G)$ has just one generator, that is, the class of Hamiltonian cycle. If the homology class of Hamiltonian cycle is trivial, we immediately have $H_1(G)$ is trivial.   

\begin{proof}
Let $\Psi$ be the collection of all homology classes mentioned in the statement and $A^i$ be the set of homology classes of all Type$i$ cycles on $G$. WLOG, it suffices to show for $i=1$.\\
\indent To show that $\Psi$ spans $H_1(G)$, it suffices to exhibit that for each 
$[(r\ze s)]\in A^1-\mathcal{T}(\mathfrak{B})^1$ can be written in a combination of elements of $\Psi$. Since $\mathcal{T}$ is a spanning tree on $G$, there is a perfect cycle $C$ on $s(\mathcal{G})$ which is formed by some edges of $\mathcal{T}$ and $\{r,s\}$. Assume that $\mathcal{T}(\mathfrak{B})$ consists of edges $\{a_1,a_2\},\{a_3,a_4\},...,\{a_{k-1},a_k\}$. By \ref{lemma}, \ref{remark3}(1) and \ref{remark1}(2), $$[C]=[(r\ze s)]+[(a_1\ze a_2)]+...+[(a_{k-1}\ze a_k)].$$ However, $C$ is a perfect cycle on $s(\mathcal{G})$, by \ref{key theorem}, $[C]=[0]$. Therefore, we have $[(r\ze s)]+[(a_1\ze a_2)]+...+[(a_{k-1}\ze a_k)]=[C]=[0]$. That is,
$$-[(r\ze s)]=[(a_1\ze a_2)]+...+[(a_{k-1}\ze a_k)].$$
By \ref{remark3}(2), we can write the previous equation in a new form as
$$-[(r\ze s)]=m[\text{Hamiltonian~cycle}]+[(a_1\ze a_2)]^1+...+[(a_{k-1}\ze a_k)]^1,$$
for an integer $m$. \newline
\indent We claim that $\Psi$ is minimal spanning set of $H_1(G)$ in order to show that it is linearly independent. Assume in the contrary that there is a spanning proper subset $\psi$ of $\Psi$. It is obvious that if the homology class of the Hamiltonian cycle is not in $\psi$, $H_1(G)$ is not able to be spanned by $\psi$. So, let it be included in $\psi$. Pick $[(u\ze v)]\in \Psi-\psi$. Then, there are some elements $[(b_1\ze b_2)],...,[(b_l\ze b_{l+1})]$ of $\psi$ such that 
\begin{equation}
\sum_{i=1}^lm_i[(b_i\ze b_{i+1})]+[(u\ze v)]=[0] \tag{$*$}  
\end{equation}
for $m_1,...,m_l\in \mathbb{Z}$. Since the combination of the homology classes in $(*)$ can always be grouped into a combination of the homology classes of perfect cycles, by \ref{key theorem}, these perfect cycles must be on $s(\mathcal{G})$. However, each $[(b_i\ze b_{i+1})]$ corresponds to an edge of the basis subgraph $\mathcal{T}(\mathfrak{B})$. Thus, from $(*)$, we finally have the perfect cycles on the tree $T$ corresponding to $\mathcal{T}$, and that contradicts to the property of trees. 
\end{proof}

\begin{thm}\label{key theorem 3}
Let $G$ be a Hamiltonian graph with a fixed circle form and $\mathcal{G}_1,...,\mathcal{G}_k$ be its all distinct nets with $\mathcal{T}_1,...,\mathcal{T}_k$ their corresponding spanning sets. The first homology class of the Hamiltonian cycle and $\bigcup_{j=1}^k \mathcal{T}_j(\mathfrak{B})^i$ form a basis of $H_1(G)$, for any $i$.
\end{thm}

\begin{proof}
Let $\Omega$ be the set of all homology classes mentioned in the statement. It is so easy by applying the proof of \ref{key theorem 2} for showing that $\Omega$ spans $H_1(G)$. So, it only remains to prove the linear independence of $\Omega$. Again, WLOG, it suffices to show for $i=1$.
\newline
\indent Assume that $|\mathcal{T}_j(\mathfrak{B})^1|=n_j$ for each $j$. Let 
\begin{equation}
h[\text{Hamiltonian~cycle}]+\sum_{j=1}^k \sum_{k=1}^{n_j} m_{jk}[(r_{jk}\ze s_{jk})]^1=[0], \tag{$*$}
\end{equation}
where, for each $k=1,...,n_j$, $[(r_{jk}\ze s_{jk})]^1\in \mathcal{T}_j(\mathfrak{B})^1$ and $h,m_{jk}\in \mathbb{Z}$.
The left side of $(*)$ can always be rearranged and grouped into a combination of the homology classes of the perfect cycles on $G$. Since it is also trivial, this can interpret that each homology class must be trivial. By \ref{key theorem}, we know that these perfect cycles lie on the graph corresponding to the nets. Additionally, each one has to travel just only on one of the graph of a net. However, each 
$(r_{jk}\ze s_{jk})$ is picked from $\mathcal{T}_j(\mathfrak{B})$. Hence, this concludes that each of these perfect cycles is the perfect cycle on the corresponding subtree of the net where it lives. This is impossible so the only possible solution for the equation $(*)$ to exist is that $h$ and each $m_{jk}$ must be zero.
\end{proof}

The following is the conclusion of the whole paper, which is obviously true by \ref{key theorem 3}
\begin{cor}
The first homology group of any Hamiltonian graphs is torsion-free.
\end{cor}

\bibliographystyle{unsrt}

\end{document}